\documentclass[12pt,centertags,reqno]{amsart}

\usepackage[foot]{amsaddr}
\usepackage[english]{babel}
\usepackage[T1]{fontenc}
\usepackage[numbers]{natbib}
\usepackage{amssymb}
\usepackage{fancyhdr}
\usepackage{url}
\usepackage{hyperref}
\usepackage{verbatim}
\usepackage{leftidx}
\usepackage{color,graphicx}

\usepackage{mathrsfs, mathtools}
\usepackage{stmaryrd}

\usepackage{marvosym}
\mathtoolsset{showonlyrefs}

\numberwithin{equation}{section} \makeatletter
\renewcommand{\subsection}{\@startsection
{subsection}{2}{0mm}{\baselineskip}{-0.25cm}
{\normalfont\normalsize\bf}} \makeatother

\newtheorem{theorem}{Theorem}[section]
\newtheorem{lemma}[theorem]{Lemma}
\newtheorem{corollary}[theorem]{Corollary}

\theoremstyle{definition}

\newtheorem{ass}[theorem]{Assumption}
\theoremstyle{remark}
\newtheorem{remark}[theorem]{Remark}
\newtheorem{example}[theorem]{Example}

\def \F {\mathcal F}

\def \P {\mathbf P}
\def \Q {\mathbf Q}

\def \R {\mathbb R}
\def \bF {\mathbb F}

\def \bE {\mathbb E}
\def \bN {\mathbb N}

\def \CAT {\textrm{CAT}}

\newcommand{\ud}{\mathrm d}

\newcommand{\esp}[2][\mathbb E] {#1\left[#2\right]}

\begin{document}

\author[K.~Colaneri]{Katia Colaneri}\address{Katia Colaneri, Department of Economics and Finance, University of Rome Tor Vergata, Via Columbia 2,
00133 Roma, IT.}\email{katia.colaneri@uniroma2.it}

\author[R.~Frey]{R\"udiger Frey}\address{R\"udiger Frey, Institute for Statistics and Mathematics, Vienna University of Economics and Business,
Welthandelsplatz, 1, 1020 Vienna, Austria }\email{rfrey@wu.ac.at}\thanks{We are grateful to Verena K\"ock and to two unknown Referees for useful comments and suggestions. The work of Katia Colaneri was partially supported by INdAM-GNAMPA through projects UFMBAZ-2019/000436 and U-UFMBAZ-2020-000791.}

\title[Classical solutions of PIDEs]{Classical solutions of the Backward PIDE  for  Markov Modulated Marked Point Processes and Applications to \\ CAT Bonds}


\begin{abstract}
 The objective of this paper is to give conditions ensuring  that the backward  partial integro differential equation   associated with  a
 multidimensional jump-diffusion with a pure jump component has a unique classical solution; that is the solution is continuous, twice
 differentiable  in the diffusion component and  differentiable in time. Our proof uses a probabilistic argument  and extends the results of
 \citet{bib:pham-98} to processes with a pure jump component  where the jump intensity is modulated by a diffusion process. This result is
 particularly useful in some applications to pricing and hedging of financial and actuarial instruments, and we provide an example to pricing of
 CAT bonds.
\end{abstract}

\maketitle


\section{Introduction}
This paper studies  conditions for the existence of  smooth solutions for certain partial integro-differential equations (PIDEs) associated with the generator of a Markov jump-diffusion process with at least one pure jump component. In concrete terms, we consider a Markov process of the form $X=(Z,L)$, where $L$ is a pure jump process and $Z$ is a general $d$-dimensional jump diffusion that
modulates the local characteristics (jump intensity and jump-size distribution) of $L$.  Given functions $c,f$ from $[0,T] \times \R^{d+1}$ to $\R$
and $g$ from $\R^{d+1}$ to $\R$, we define the  function $v$ by
\begin{equation} \label{eq:def-v-intro}
v(t,x)= \esp{\int_t^T e^{-\int_t^s c(u, X_u) \ud u} f(s, X_s) \ud s + e^{-\int_t^T c(s, X_s) \ud s} g(X_T)\mid X_t=x}.
\end{equation}
The goal of the paper  is to give regularity conditions on  $c$, $f$ and $g$ and on the generator $\mathcal L$ of the process  $X$  which ensure that
$v$ is a classical solution (i.e. Lipschitz continuous, $\mathcal C^2$ in $z$, $\mathcal{C}^1$ in $t$)  of the backward
PIDE associated with  $\mathcal{L}$.

Jump processes with characteristics modulated by a (jump) diffusion  are frequently used  in non-life insurance,  for instance in the valuation of loss layers  or {\em catastrophe bonds}  (CAT bonds), and in specific  fields of finance such as credit risk modelling  or tick data models for  high frequency trading. Functions $v$ of the form  \eqref{eq:def-v-intro} arise  in  pricing and hedging problems, where the function $f$ may represent instantaneous dividend payments, function $c$ the discount rate and function $g$ the terminal  payoff. The fact that $v$ is a \emph{classical} solution  of the backward  PIDE (and not merely a viscosity solution) is
essential in this context, as it allows to compute the  \emph{martingale representation} of the {price process}  $(v(t, X_t))_{t \ge 0}$ and, consequently, hedging strategies (see, e.g. \citet{bib:frey-00b}, or the discussion in Section \ref{sec:CAT} below).
To further illustrate the financial relevance of our results we investigate in detail the pricing and hedging of a CAT bond. CAT bonds are obligations issued
by (re)insurance companies to transfer to financial markets  the risk of low-frequency high-severity events such as natural catastrophes; if
cumulative losses due to such events exceed a (high) threshold, coupons and face value of the bond are reduced. We show that, in a typical actuarial
model, the pricing  of CAT bonds leads to a PIDE with a pure jump component, and that the hedging requires that the PIDE has a classical solution.

While the PDE characterization of  $v$ is  well understood in the case of diffusion processes, where $\mathcal{L}$ is a second-order differential operator, there are only a few contributions that study the jump-diffusion case, that is when  $\mathcal{L}$ is an integro differential operator, see  for instance \citet{bib:gihman-skohorod-80, bib:bensoussan-lions-82, bib:pham-98, bib:davis-lleo-13}. Moreover, this limited literature does not cover the case where $X$ has a pure jump component, and it makes quite restrictive regularity assumptions on the coefficients appearing in $\mathcal{L}$. Most relevant for our analysis are the results of \citet[Section~5]{bib:pham-98}, where existence and uniqueness of
a smooth solution for the backward PIDE is obtained in the case where the process $X$ can be written as the solution of an SDE driven by a Brownian motion and
an exogenous Poisson random measure. His analysis relies on two strong assumptions:  first, the coefficients in the SDE representation of $X$
satisfy a strong Lipschitz assumption; second, the diffusion part of $\mathcal{L}$ is uniformly elliptic.  While these two assumptions look quite natural, there are  many practically relevant situations where they are not met. For instance, the assumptions that the jumps of the state process $X$ are driven by a Poisson random measure with Lipschitz jump size coefficients excludes models that employ (compound) Cox  processes (essentially marked point processes with stochastic jump intensity). Moreover, the ellipticality assumption on the instantaneous covariance matrix   of $X$ implies that the diffusion part cannot be degenerate in any direction and hence it excludes  processes with a pure jump component. These points are clarified in more details in Remark \ref{rem:Pham}.

The goal of this paper is therefore, to extend the results of \citet[Section~5]{bib:pham-98} to the more general situation  where the multidimensional jump-diffusion process $X$ may have a pure jump component, and to weaken  the regularity assumptions on the integral part of the generator $\mathcal{L}$ of $X$, so to include the case where the jumps of $X$ are described by a random measure with Markov modulated compensator. Our approach is based on a change of measure argument:  loosely speaking we start from a reference probability space where the local characteristics of $L$ are deterministic and we revert to the original model by changing probability. We consider the extended state process given by the pair $(X, \xi)$, where $\xi$ is the martingale density of the measure change. Under the reference probability, the new state process  falls under the {\em viscosity} modeling framework of \citet{bib:pham-98}. Using Bayes formula and the results of
\citet{bib:pham-98} we then obtain that $v$ is the unique viscosity solution of the backward PIDE associated with the operator $\mathcal L$.  Finally, in order to show that $v$ is also a classical solution,   weneed to apply a fixed point argument.

The reminder of the paper is organized as follows. In Section \ref{sec:pb} we introduce the problem and the main assumptions. In
Section~\ref{sec:CAT} we discuss the pricing and hedging of CAT bonds. In Section \ref{sec:change} we construct the process $X$ via change of measure.
Finally,  we prove existence and uniqueness for the solution to the  backward PIDE in Section \ref{sec:theorem}.

\section{Modeling framework and problem formulation}\label{sec:pb}

We fix a  probability space $(\Omega, \F, \P)$, a time horizon $T$ and a right continuous and complete filtration $\bF$. Consider measurable
functions $a:[0,T]\times \R^d\times \R \to \R^d$ and $b:[0,T]\times \R^d\times \R \to \R^{d\times d}$, $\gamma^Z:[0,T]\times \R^d\times \R\times E
\to \R^d$ and $\gamma^L:[0,T]\times \R^d\times \R\times E \to \R$, where $(E, \mathcal E)$ is a separable measurable  space. Define the symmetric
matrix $\Sigma(t,z,l) = (\sigma_{i,j}(t,z,l),  \  i,j=1,\dots, d) $,  by
$$ \Sigma(t,z,l) = b(t,z,l)b^\top (t,z,l).$$
Throughout the paper we use the following notation for partial derivatives: for every function $h:[0,T] \times  \R^d \times \R \to \R$ which is
$\mathcal C^2$ in $z$ and continuous in $l$,  we write $ h_{z_i}$ for the first derivatives of $h$ with respect to $z_i$ for $i \in \{1, \dots d\}$
respectively, $h_{z_i,z_j}$  for second derivatives, for $i,j\in \{1, \dots, d\}$,  and finally  $h_t$ denotes  the first derivative with respect to
time.

We assume that $(\Omega, \F, \P)$ supports a RCLL process $X=(Z,L)$, which is the unique solution of the martingale problem associated with the
following (time-inhomogeneous) integro-differential operator $\mathcal L_t$
\begin{equation}\label{eq:generator_P}
\begin{split}
\mathcal{L}_t \varphi(x)=\mathcal{L}_t \varphi(z,l):=&\sum_{i=1}^d a_i(t,z,l) \varphi_{z_i}(z,l)+\frac{1}{2}\sum_{i,j=1}^d \sigma_{i,j}(t,z,l)
\varphi_{z_i, z_j}(z,l) \\
&+  \int_E \big(\varphi(z+\gamma^Z(t,z,l,u), l+\gamma^L(t,z,l,u))-\varphi(z,l)\big ) \nu(t, x;\ud u)\,,
\end{split}
\end{equation}
for every $(z,l)\in \R^d\times \R$, $t \in [0,T]$ and every function $\varphi: \R^d\times \R\to \R$ which is $\mathcal C^2$ in $z$ and continuous in
$l$, bounded with bounded derivatives.

It will be shown later  that under some regularity conditions (specifically
Assumptions~\ref{ass:regularity} and \ref{ass:problem2}) the unique solution to the martingale problem for the generator $\mathcal{L}_t$ in
\eqref{eq:generator_P} exists, see Corollary~\ref{cor-martingale-problem} below.
This result is  relevant from both a theoretical and also an applied point of view. In fact,  problems involving Markov processes $X=(Z,L)$ with a generator of the form \eqref{eq:generator_P} are largely used  in an actuarial context   (see, e.g.
\citet{bib:grandell-91,bib:ceci-hedging2015}) and in a financial context (see, e.g. \citet{bib:bielecki-rutkowski-02a, bib:cartea-jaimungal-penalva-15,
bib:colaneri-eksi-frey-szolgyenyi-19,bib:frey-runggaldier-01}); a specific application is discussed in Section~\ref{sec:CAT} below.

To illustrate our setup  we now give the generator for two simple  examples.
\begin{example}\label{examples}
First, we consider the case  where $L$ is a time-homogeneous Cox process with intensity $\lambda(Z_t)$, where $Z$ is a one dimensional diffusion
following the dynamics $\ud Z_t = a(Z_t) \ud t + b(Z_t) \ud W_t$, for a  Brownian motion $W$.  The generator $\mathcal{L}$ of the process $X=(Z,L)$
reads as
\begin{equation}
\mathcal L \varphi(x)=\mathcal{L} \varphi(z,l)= a(z)\varphi_{z}(z,l)+\frac{1}{2} b^2(z) \varphi_{z,z}(z,l)+  [\varphi(z, l+1)-\varphi(z,l)]
\lambda(z).
\end{equation}
Second, assume more generally  that  $L$ is  a time-homogenous compound Cox process with jump intensity $\lambda(Z_t)$ and jump size distribution
$\mu(\ud u)$ on $\R$. We still assume that $Z$ is a one dimensional diffusion of the same type as before. In this case the  generator $\mathcal L$
of $X=(Z,L)$ has the form
\[
\mathcal L \varphi(z,l)=a(z) \varphi_z(z,l)+ \frac{1}{2}b^2(z) \varphi_{zz}(z,l)+ \int_{\R}\left(\varphi(z,
l+u)-\varphi(z,l)\right)\lambda(z)\mu(\ud u);
\]
in particular $\nu(z,l;\ud u) = \lambda(z) \mu(\ud u)$.  Note that our setup goes beyond compound Cox processes presented in these two examples,
since the form of the generator  in \eqref{eq:generator_P} encompasses also  models with joint jumps in $L$ and $Z$. This feature can be useful to model self-exciting phenomena.
\end{example}

\vspace{1ex}

We continue with the problem formulation.  Let $g:\R^{d+1}\to \R$ be a payoff function, $f:[0,T]\times\R^{d+1}\to \R$ a dividend rate function  and
$c:[0,T]\times\R^{d+1} \to \R$ a discount rate.
In the reminder of the paper we work under the following assumptions.

\begin{ass}\label{ass:regularity}
\begin{itemize}
\item[]
\item[(A0)] The functions $a,b,\gamma^Z, \gamma^L, f, g$ and $c$ are continuous.
\item[(A1)]  The functions $a$ and $b$ are locally Lipschitz in $(t,x)$ and Lipschitz in $x$ for all $t \in [0,T]$.
\item[(A2)] There exists a finite measure $\widetilde \nu(\ud u)$ on $(E, \mathcal E)$  such that the  measure $\nu(t,x;\ud u)$ is equivalent to
    $\widetilde \nu(\ud u)$; the Radon Nikodym derivative $\nu(t,x,u):= (\ud \nu(t,x)/\ud \widetilde \nu) (u)$ satisfies $\nu(t,x,u) \le 1$ for
    all  $(t,x,u) \in [0,T] \times \R^{d+1}\times E$.
\item[(A3)]  The functions $\gamma^Z$, $\gamma^L$ and  $\nu$ satisfy for all $t \in [0,T]$ and $x,y\in \R^{d+1}$ .
\begin{gather}
|\gamma^Z(t,x,u)-\gamma^Z(t,y,u)|+|\gamma^L(t,x,u)-\gamma^L(t,y,u)| + |\nu(t,x,u) - \nu(t,y,u)| \leq \rho(u)|x-y|,\\
|\gamma^Z(t,x,u)|+|\gamma^L(t,x,u)|\leq \rho(u),
\end{gather}
where the function $\rho:E\to \R^+$ is such that $\int_E\rho^2(u) \widetilde \nu(\ud u)<\infty$.
 \item[(A4)] The function $c$ is bounded and locally H\"{o}lder continuous.
 \item[(A5)] The functions $g$ and $f$ are bounded and satisfy for every $t,s\in [0,T]$ and $x,y\in \R^{d+1}$
\begin{align}
|f(t,x)-f(s,y)|+|g(x)-g(y)|\leq K(|t-s|+|x-y|).
\end{align}
\end{itemize}
\end{ass}

Note that  by Assumption (A2), $\sup_{(t,x)\in [0,T]\times \R^{d+1}}\nu(t,x,E)<\widetilde{\nu}(E)<\infty$, that is (A2) implies  that the jump
intensity of $L$ is bounded.

For  \emph{fixed} $l \in \R$ we now introduce the differential operator ${\mathcal L}^*$  by
\begin{align}\label{eq:generator*}
{\mathcal L}_{(t,l)}^* \varphi(z)=&\sum_{i=1}^d a_i(t,z,l) \varphi_{z_i}(z)+\frac{1}{2}\sum_{i,j=1}^d \sigma_{i,j}(t,z,l) \varphi_{z_i, z_j}(z)
\end{align}
for every $z\in \R^d$ and $t \in [0,T]$ and every function $\varphi: \R^d\to \R$ with $\varphi \in \mathcal C^2$, bounded with bounded derivatives.
Loosely speaking, ${\mathcal L}^*$ is the diffusion part of $\mathcal{L}$.

\begin{ass}\label{ass:problem2}
For every fixed $l \in \R$ and every bounded and Lipschitz continuous function $F:[0,T]\times\R^{d+1}\to \R$  the Cauchy  problem
\begin{align}
&\psi_t(t,z,l)+{\mathcal L}_{(t,l)}^* \psi(t,z,l) + F(t,z,l) = c(t,z,l) \psi(t,z,l)   \, \quad (t,z)\in[0,T]\times \R^d,\\
&\psi(T,z,l) = g(z,l), \quad z \in \R^d,
\end{align}
has a unique bounded classical solution.
\end{ass}
Sufficient conditions for Assumption~\ref{ass:problem2} to hold are given, for instance, in \citet[Chapter~1]{bib:friedman-64}. They amount to
assuming further to (A0)--(A5) in Assumption \ref{ass:regularity} that the functions $a(t,z,l)$ and $b(t,z,l)$ are bounded
and that the matrix $\Sigma(t,z,l)$ is uniformly elliptic in $z$ on $\R^d$, that is, there exists $\widetilde C>0$ such that for any $\zeta \in
\R^d$, $\zeta^\top \Sigma(t,z,l) \zeta\geq \widetilde C \|\zeta\|^2$ for every $(t, z,l)\in [0,T]\times \R^{d}\times \R$.
In the case where $\mathcal L^*$ is the generator of an affine diffusion (not necessarily strictly elliptic), existence and uniqueness of the
solution of the Cauchy problem in Assumption \ref{ass:problem2} is discussed, for instance in \citet{cordoni2013transition}.

The  following theorem is the main result of the paper.

\begin{theorem}\label{thm:smooth_soln}
Let (A0)--(A5) in Assumption \ref{ass:regularity} and Assumption \ref{ass:problem2} hold. Then the function $v$ given by
\begin{equation}\label{eq:value_function}
v(t,z,l)=\esp{\int_t^T e^{-\int_t^s c(u, Z_u) \ud u} f(s, Z_s, L_s) \ud s + e^{-\int_t^Tc(s, Z_s) \ud s} g(Z_T, L_T)\mid (Z_t, L_t)=(z,l)},
\end{equation}
is bounded, continuous on $[0,T]\times \R^d \times \R$, Lipschitz in $x=(z,l)$ uniformly in $t$ and, for fixed $l$, $\mathcal{C}^1$ in $t$ and
$\mathcal{C}^2$ in $z$. Moreover $v$ is a classical solution of the Cauchy problem
\begin{align}
&v_t(t,z,l)+\mathcal{L}_tv(t,z,l)+f(t,z,l) = c(t,z,l)v(t,z,l), \quad (t,z,l)\in[0,T]\times \R^d\times \R, \label{eq:PDE_w_integral}\\
&v(T,z,l)=g(z,l), \quad (z,l)\in \R^d\times\R.
\end{align}
\end{theorem}
Note that in Theorem \ref{thm:smooth_soln}, we obtain different degrees of regularity  in $z$ and $l$:  $v$ is only Lipschitz in $l$ but $\mathcal
C^2$ in $z$. This reflects that there is a diffusion component in $Z$ so that the transition kernel of $X$ has a smoothing effect in the $z$
direction, whereas no such smoothing can be expected in the $l$ direction\footnote{{Although the reader may convey that the process $X$ is essentially a multidimensional jump-diffusion, with a possibly degenerate diffusion part, it is fundamental in our analysis to disentangle its (non-degenerate) jump-diffusion part $Z$ and its pure jump part $L$. Indeed, the process $Z$ allows to define the operator $\mathcal L^*$, and hence the problem in Assumption 2.3 for which existence and uniqueness of the solution is retrieved by classical literature, and for instance, implied by the property that the diffusion coefficient is uniformly elliptic. The main consequence of this fact is that the value function in the $z$ component inherits a higher degree of regularity due to the smoothing effect of the diffusion. A pure jump process $L$, instead, cannot bring more regularity than continuity to the value function in the component $l$.}}. { The proof of the theorem is given in Section~\ref{sec:theorem} and uses the results on viscosity solutions of \citet{bib:pham-98}. However, these results cannot be applied directly, as our setting does not fall under the assumptions of \citet{bib:pham-98}, particularly due to the fact that the jump part of the process $X$ is not driven by a Poisson random measure with Lipschitz coefficients but instead we have a random measure with Markov modulated compensator. Hence an intermediate step is needed. Precisely, we will construct
the model  via a change of measure,  in the same spirit as the reference probability approach in nonlinear filtering. This is discussed in
Section~\ref{sec:change}.}

\begin{remark}[Contribution of the paper]\label{rem:Pham}
In order to clarify the contribution of our paper we  compare our setting with that of \cite{bib:pham-98}. Section 5 of that paper is dedicated to the analysis of smooth solutions for the Cauchy problem associated to a linear parabolic partial-integro differential operator.
There are two reasons why we cannot directly apply those results.
First, \cite[Section 5]{bib:pham-98} assumes that the diffusion part of the process $X$ is uniformly elliptic (Assumption (H0) of that paper). Here we relax this hypothesis and consider the case where the  diffusion part  can be degenerate in some direction: indeed, ellipticality is needed only in the component $Z$ (see the discussion after Assumption \ref{ass:problem2}), and hence the process $X$ may have a pure jump component as well, which is represented by the process $L$. Second, in our paper we relax the   regularity assumptions on the jump coefficients in both $Z$ and $L$. As in \citet{bib:pham-98}, we consider  jump size functions $\gamma^Z$ and $\gamma^L$ that are Lipschitz. However, we do not restrict to the case where the jump times are generated by a Poisson random measure with deterministic compensator but allow also for stochastic jump intensities.  In fact, although it might be possible to write our general jump measure in terms of a Poisson random measure, this induces a transformation on the jump size coefficients that would no longer  satisfy the  Lipschitz conditions. We elaborate on this point  in Example~\ref{ex:cox-continued} below.  We would like to underline once more that our extensions are relevant for applications in insurance and finance, since problems involving jump-diffusion state processes with degenerate diffusions in some directions and pure jump components driven by random measures with Markov modulated compensators are frequently used in these fields.
\end{remark}

\begin{example}[Example~\ref{examples} continued] \label{ex:cox-continued} We now give  conditions ensuring that  Theorem~\ref{thm:smooth_soln} applies to the case of a
compound Cox process.  Assume that the functions $a(\cdot)$, $b(\cdot)$ and $\lambda(\cdot)$ are Lipschitz and  that $\lambda(z) \le \overline
\lambda < \infty$ for all $z \in \R$. In that case the generator $\mathcal{L}$ of $X=(Z,L)$ satisfies the regularity conditions from
Assumption~\ref{ass:regularity}. In fact, we may choose $E= \R$, $\gamma^Z \equiv 0$, $\gamma^L(t,z,u)=u$ and the reference measure $\widetilde \nu
(\ud u) =\overline \lambda \mu(\ud u)$ so that the density $\ud \nu(l,z;\ud u) /\ud \widetilde\nu (\ud u)$ is given by the Lipschitz function
$\nu(z,l,u) = \lambda(z)/{\overline{\lambda}} \le 1$. A sufficient condition for Assumption~\ref{ass:problem2} to hold is that the functions
$a(\cdot)$ and $b(\cdot)$ are bounded and that $b^2(z) \ge \underline{b}$ for some $ \underline{b} >0$.
Next, we explain why a compound Cox process is not covered by the original results of \citet[Section~5]{bib:pham-98}. \citet{bib:pham-98}
considers an integro differential operator with integral term of the form
\begin{equation} \label{eq:integral-term}
 \int_E [\varphi(z+\gamma^Z(z,l,u), l+\gamma^L(z,l,u))-\varphi(z,l)]\nu(\ud u)
\end{equation}
for a finite measure $\nu(\ud u)$ that is \emph{independent} of the state $x$ of the process, and assumes that the functions $\gamma^Z$
and $\gamma^L$ satisfy the Lipschitz condition (A3) from Assumption~\ref{ass:regularity}. It is indeed possible to write the integral part of the
generator of a compound Cox process in the form \eqref{eq:integral-term} if we construct $L$ via thinning. For doing this, we start from a compound
Poisson process $\overline L$ with constant jump intensity $\overline \lambda$ and jump size distribution $\mu$. At each jump time $\overline T_n$
of $\overline L$  we independently sample a standard uniform random variable $V$ and we retain the jump if $V \le
\frac{\lambda(Z_{T_n})}{\overline{\lambda}}$. This corresponds to a representation \eqref{eq:integral-term} for the integral part of the generator
of $X$: take    $E = [0,1] \times \R$ with elements $u=(v,w)$, let   $\nu(\ud u)  = \nu (\ud v \ud w)=  \overline{\lambda} \ud v \mu(\ud w)$, and
put
$$ \gamma^L(z,u) = \gamma^L(z,v,w) = w 1_{\{v \le \frac{\lambda(z)}{\overline{\lambda}}\}}\,.$$
Note however, that the function $\gamma^L$ is not continuous and, in particular, it does not satisfy the Lipschitz condition (A3). Hence, the
results of \citet{bib:pham-98} do not apply to a compound Cox process, not even his results on viscosity solutions, where the fact that $L$ is a
pure jump process is not an issue.
\end{example}

\section{Pricing and hedging of a CAT bond}\label{sec:CAT}

In order to show that our results are relevant in insurance and finance we now discuss the problem of pricing and hedging of a {\em catastrophe
bond} (CAT bond). CAT bonds are obligations with short maturities, (usually one to three years) issued by insurance and reinsurance companies to
transfer to financial markets the risk of extreme  losses in non-life  insurance.  The payoff of a CAT bond  depends on some underlying loss index
$L$ that measures the losses in a given pool of  insurance contracts with  specified  loss type, geographical loss area and reporting period.  If
the loss index stays below a given threshold,  investors (the buyers of the CAT bond) receive coupons  and the  face value at maturity; if the loss
index is higher than the threshold, coupon payments are reduced and the face value is repaid only partially. Buyers accept the additional risk for a
generous rate of return.

There exist many variations of CAT bonds, see for instance \citet{bib:cox-2000, bib:lee-2004, bib:jarrow2010, bib:jaimungal-14}. In this paper we
model the loss index as a pure jump process $L$.  We ignore coupon payments for simplicity, and we assume that at the payoff of the bond at the
maturity $T$ is given by the face value $F$, reduced by the payoff of a loss layer on $L_T$ with attachment points  $K_1$ and $K_2 = K_1 + \delta F$
for some $\delta \in [0,1]$. Formally, the payoff at maturity is thus given by $g(L_T)$ for the  bounded and Lipschitz continuous function
\begin{equation} \label{eq:CAT-payoff}
g(l) = F- \Big ( (l- K_1)^+ - (l -K_2)^+ \Big);
\end{equation}
in particular, for $L_T \le K_1$ the bond pays the face value in full and for $L_T >K_2 $ the bond has payoff $(1-\delta)F$. The percentage $\delta$
plays the same role as the loss given default in credit risk.
The goal is to determine the price of the CAT bond (in Section \ref{sec:pricing}) and a self-financing hedging strategy that allows to cover for the
interest rate risk (in Section \ref{sec:hedging}). Due to market incompleteness we address the hedging problem via a quadratic approach. Under
standard assumptions on the model dynamics the price of the bond is given by the solution of a PIDE with a pure jump component as
in~\eqref{eq:PDE_w_integral}. We show that a classical solution of this PIDE is needed to determine the hedging strategy, as the latter
involves the use of derivatives of the pricing function.

\subsubsection*{Comments.} Of course other payoff functions than the one in \eqref{eq:CAT-payoff} could be considered as well, provided that they are consistent with Assumption~A5. For instance,  \citet{bib:jaimungal-14} considers call and put options on the loss index\footnote{The payoff of a call option is unbounded, but  call options can usually be replaced by  put options via put-call parity.}.  More generally, our analysis could also be extended to  models as in \citet{bib:jarrow2010} where the ``default time'' $\tau$ of a CAT bond (the time when coupons and the repayment of the face value are reduced) is modeled as a doubly stochastic random time whose intensity might depend on the loss index.

\citet{bib:jaimungal-14} characterize the price of the CAT bond in terms of a  PIDE  that is similar to our pricing PIDE (Equation ~\eqref{eq:cat-pide} below). However
they simply assume the existence of a smooth solution and apply Fourier transform to compute the solution numerically. We mention a few other hedging problems in insurance and finance  where smooth solutions of a PIDE of the type \eqref{eq:PDE_w_integral} play a crucial role:      \citet{bib:ceci-hedging2015} consider
classical solution of a PIDE to determine  the hedge ratio of unit-linked life insurance contracts under partial information;  \citet{bib:ceci-colaneri-frey-koeck-20} study the hedging  of reinsurance counterparty credit risk; the  hedging of derivatives in a high frequency data setting is considered in \citet{bib:frey-runggaldier-01}. In all these papers  the existence of a classical solution of the PIDE is assumed and not established. By giving sufficient conditions for the existence of a smooth solution to the pricing PIDE our paper puts the computation of hedge ratios in these papers on a sound mathematical footing.

\subsection{The pricing problem}   \label{sec:pricing}

We introduce the dynamics of the loss index and of the short rate of interest. For this we fix a probability space $(\Omega, \F, \Q)$ with  a
complete and right continuous filtration $\bF$, and we interpret $\Q$ as the pricing measure. Markets for CAT bonds are incomplete, and therefore, the
choice of the  pricing measure $\Q$ is a  delicate modelling issue involving also the real-world measure $\P$; see for instance the discussion in
\citet{bib:ceci-colaneri-frey-koeck-20}. However, this question is not central to the present paper so that we specify the dynamics of all model
quantities directly under $\Q$. We also fix a time horizon $T$ and assume that $\F=\F_T$.

We begin with the loss index. We assume that the loss index is modelled as
\[
L_t=\sum_{n=1}^{N_t} U_n
\]
and that the loss amounts $(U_n)_{n \in \bN}$ are given by a sequence of independent identically distributed nonnegative random variables with
distribution $\mu(\ud u)$ on $[0, \infty)$.
Following \citet{bib:jaimungal-14} we assume that the $N=(N_t)_{t \in [0,T]}$ is a point  process with  intensity
$\lambda(t,Z_{t}^1)$ for a positive, bounded and Lipschitz function  $\lambda:[0,T]\times \R^+ \times \R \to (0, \bar \lambda]$, with $\bar \lambda
>0$. The factor $Z^1=(Z^1_t)_{t \in [0,T]}$ that affects the intensity has dynamics
\begin{equation} \label{eq:dZ-CATbond}
\ud Z_{t}^1= a_1(b_1-Z_t^1)\ud t + \sqrt{1-\rho^2} \sigma_1 \ud W_t^1 + \rho \sigma_1 \ud W^2_t + \gamma^Z \ud L_t
\end{equation}
for independent $\Q$ Brownian motions $(W^1_t)_{t \in [0,T]}$ and $(W^2_t)_{t \in [0,T]}$ and constants $a_1,b_1,\gamma^Z \ge 0$, $\sigma_1 >0$, $\rho \in (-1,1)$. Note that
\eqref{eq:dZ-CATbond} allows for exogenous fluctuations in the loss intensity (due to the diffusion term)  and for \emph{self-excitation}: a loss
event (a jump of $L$) causes an upward jump in the loss intensity which in turn raises the likelihood of future losses.   This effect is dampened
over time  by the mean reverting drift. The intensity dynamics  illustrates the rich modelling possibilities under Theorem~\ref{thm:smooth_soln}.

We model the short rate of interest by $r_t=r(Z_t^2)$ for an increasing bounded and Lipschitz continuous function  $r:\R \to \R$. For concreteness  we
assume that $Z^2=(Z^2_t)_{t \in [0,T]}$ follows a Vasicek model
\[
\ud Z_t^2=a_2(b_2-Z_t^2)\ud t + \sigma_2 \ud W_t^2, \quad Z_0^2 \in \R^+, ,
\]
with parameters $a _2,b_2,\sigma_2>0$. Moreover we fix two constants $\underline r < \overline r$ and we let $r(z) = (z \vee \underline{r} )\wedge\overline{r}$. Assuming that the factor processes $Z^1$ and $Z^2$ are correlated depicts a certain dependence between losses and the market. Such
effects are empirically observed: for instance natural catastrophes, or unexpected events such as the recent Covid-19 pandemic, affect both the loss index (typically by increasing the intensity of loss events) and the performance of financial markets.

We assume that a riskless zero coupon bond with maturity $T$ and a CAT bond with maturity $T$ and payoff $g(L_T)$ are traded. The price of the zero coupon bond is given by
$P_t =  \bE^\Q \big [ e^{-\int_t^T r(Z_s^2) \ud s} \mid \F_t \big ] $. Using Theorem~\ref{thm:smooth_soln}  we get that $P_t = p(t,Z_t^2) $ for a function $p(t,z_2)\in \mathcal C^{1,2}([0,T]\times \R)$  which solves
\begin{align}\label{eq:bond-price}
p_t(t,z_2)+\mathcal{L}_t^{Z^2}p(t,z_2)=r(z_2) p(t,z_2), \quad (t, z)\in [0,T)\times \R, \quad \mbox{ and } \  p(T,z_2)=1,
\end{align}
where $\mathcal{L}_t^{Z^2}\varphi(t,z_2)=a_2(b_2-z_2) \varphi_{z_2}(t,z_2)+\frac{1}{2}\sigma_2^2 \varphi_{z_2,z_2}(t,z_2)$, is the
generator of the process $Z^2$\footnote{Notice that this result is covered by the classical Feymann-Kac formula, as the process $Z^2$ follows a one dimensional diffusion with regular coefficients.}.
Moreover, we have the bond price dynamics
\begin{align}\label{eq:bond}
\ud P_t= P_t r(Z_t^2) \ud t + P_t \beta(t,Z_t^2) \ud W^2_t
\end{align}
where $\beta(t,Z_t^2)  = \sigma_2 p_{z_2}(t,Z_t^2)/p(t,Z_t^2)$. We underline that, for $\underline{r} $ small and $\overline{r}$ large we may identify  $Z_t^2$ and
$r_t$: the function $p(t,r)$ can be approximated by the explicit bond price formula in the Vasicek model and $\beta(\cdot)$ depends only on time;
see \citet[Ch. 5 and Ch. 10]{bib:filipovic-2009}.

Let $X=(Z,L)$, where $Z=(Z^1, Z^2)$ is the two dimensional factor process. By risk-neutral pricing the price of the CAT bond is given by
\begin{align}
P^{\CAT}_t=p^{\CAT} (t,X_t) =: \bE^\Q \Big [ e^{-\int_t^T r(Z_s^2) \ud s} g(L_T) \mid \F_t \Big ]\,.
\end{align}
To identify   the PIDE that characterizes the function $p^{\CAT}$ via Theorem~\ref{thm:smooth_soln} we first determine  the generator of $X$.
Let $z=(z_1,z_2)$ and define the measure   $\nu (t, z) $ on $\R$ by $\nu (t, z;\ud u) = \lambda(t, z_1) \mu(\ud u)$.  Then the process $X$ is Markov with generator
\begin{align}
\mathcal L_t \varphi(x) &= a_1(b_1-z_1)\varphi_{z_1}(z,l) + \frac{\sigma^2_1}{2}\varphi_{z_1,z_1}(z,l) +
a_2(b_2-z_2)\varphi_{z_2}(z,l) + \frac{\sigma^2_2}{2}\varphi_{z_2,z_2}(z,l)  \\&+ \rho\sigma_1\sigma_2\varphi_{z_1,z_2}(z,l) +
\int_{[0, \infty)}\! \left(\varphi(z_1+\gamma^Z u,z_2,l+u)-\varphi(z_1,z_2,l)\right) \nu (t, z,\ud u)\,.
\end{align}
By Theorem~\ref{thm:smooth_soln} the function $p^{\CAT}$ is the unique classical solution of the PIDE
\begin{align}
p^{\CAT}_t (t,z,l) & +\mathcal{L}_t p^{\CAT}_t (t,z,l) =  r(z_2) p^{\CAT}_t (t,z,l) \quad (t,z,l)\in[0,T]\times \R^2 \times [0,
+\infty),\label{eq:cat-pide}\\
p^{\CAT}(T,z,l) &= g(l), \quad (z,l)\in \R\times[0, +\infty). \label{eq:cat-final-cond}
\end{align}

\subsection{The hedging strategy} \label{sec:hedging}
In this section we address the problem of finding a hedging strategy for the CAT bond that allows the bond holder to eliminate the interest rate risk of the bond. We recall that the
market is incomplete since there are no financial instruments that allow to hedge the risk due to insurance  losses. In this setting we apply a
quadratic hedging criterion, namely, {\em mean variance hedging}, that allows to identify the unique self-financing strategy that covers the interest rate
risk and minimizes the difference between the portfolio value and the value of the CAT bond at maturity in $L^2$-sense. We now formalize the hedging problem.
Let $(h_t)_{t \ge 0}=(h^0_t, h^1_t)_{t \ge 0}$ be a self financing strategy, where $h^0$ represents the investment in the money market account with
the price $(P^0_t)_{t \ge 0}$ and $h^1$ is the number of shares invested in the bond with the price $(P_t)_{t \ge 0}$. The {\em discounted} value of
the strategy $h$ is given by
\[
\widetilde V_t(h)=v_0+\int_0^th^1_s \ud \widetilde P_s
\]
with $\widetilde P=P/P^0$ being the discounted value of the bond.
A self financing strategy $h$ is admissible if it is $\bF$-predictable and satisfies $\mathbb{E}\left[\int_0^T (h^1_t)^2 P_t^2 \ud t
\right]<\infty$\footnote{
This condition guarantees that the discounted value process $V(h)$ is a square integrable martingale.}.

Let $H=g(L_T)/P^0_T$ be the discounted payoff of the CAT bond.
The hedging problem consists on finding a self-financing strategy $h^*=(h^{0*}, h^{1*})$ with initial value $v_0$ which minimizes the {\em quadratic hedging error}
\[
\bE^\Q\left[\left(H -\widetilde V_T(h)\right)^2\right],\label{eq:error}
\]
over the set of all admissible strategies.  The minimizer $h^*$ is called the {\em
mean variance hedging strategy}. Since the process $\widetilde P$ is a square integrable martingale, it is well known (see for instance \citet{bib:follmer-86} or
\citet{schweizer2001guided}) that the optimal strategy can be determined using the Galtchouk--Kunita--Watanabe decomposition of the discounted CAT
bond price process $\widetilde P^{\CAT}= P^{\CAT}/P^0$ with respect to the discounted zero coupon bond price $\widetilde P$. This decomposition  is given by
\begin{align}\label{eq:GKW}
\widetilde P^{\CAT}_t=P^{\CAT}_0+\int_0^t \theta_s \ud \widetilde P_s + O_t
\end{align}
where $(O_t)_{t \in [0,T]}$ is a martingale null at $t=0$ and orthogonal to $\widetilde P$\footnote{Two martingales $M^1$ and $M^2$ are said to be orthogonal if the product $M^1M^2$ is  a martingale, or , equivalently, if  $M^1$ and $M^2$ have zero predictable quadratic covariation, i.e. $\langle
M^1, M^2\rangle_t=0$ for all $t \ge 0$.}.
The process $h^{1*}_t=\theta_t$ for all $t \in [0,T]$, provides the mean variance hedging strategy.
Moreover, at time $T$ it holds that $\widetilde P^{\CAT}-\widetilde V_T(h)=O_T$, so that  $O_T$ represents the hedging error.

To characterize the hedging strategy $(\theta_t)_{t \in [0,T]}$, we derive the martingale representation of the (discounted) CAT bond price process. Recall that
$P^\CAT_t=p^{\CAT}(t,Z_t,L_t)$, where $p^\CAT(t,z,l)$ satisfies the PIDE \eqref{eq:cat-pide}, with the final condition \eqref{eq:cat-final-cond}.
Hence,
\begin{align}
\frac{\ud \widetilde P^{\CAT}_t}{\widetilde P^{\CAT}_{t-}} =& \sqrt{1-\rho^2}  \beta^{\CAT}_1(t, Z_t, L_t) \ud W^1_t +  \left(\rho
\beta^{\CAT}_1(t, Z_t, L_t)+ \beta^{\CAT}_2(t, Z_t, L_t)\right) \ud W^2_t \\
&+ \int_{[0,+\infty)} \left(\alpha^{\CAT}(t, Z^1_t +\gamma^Z u,Z^2_t,L_{t^-}+u)-1\right) \left(m(\ud t, \ud z)-\nu (t,
Z^1_{t^-}, Z^2_{t},\ud u)\right)\,,
\end{align}
where $\beta^{\CAT}_i(t, z,l)=\sigma_i \ p^\CAT_{z_i}(t, z,l)/p^\CAT(t, z,l)$ for $i=1,2$, $\alpha^{\CAT}(t,z_1+\gamma, z_2, l+u)=p^\CAT
(t,z_1+\gamma, z_2, l+u)/p^\CAT(t, z_1, z_2, l)$, and the measure $m(\ud t, \ud u)$ denotes the jump random measure of the process $L$.
Since  $\ud \widetilde P_t = \widetilde P_t \beta(t,Z_t^2) \ud W^2_t $ we get
\begin{align}\label{eq:martingale_representation}
\ud \widetilde P^{\CAT}_t=& \frac{\widetilde P^{\CAT}_t}{\widetilde P_t \beta(t,Z_t^2)} \left(\rho \beta^{\CAT}_1(t, Z_t, L_t)+ \beta^{\CAT}_2(t,
Z_t, L_t) \right) \ud \widetilde P_s + \ud M_t
\end{align}
where $(M_t)_{t \in [0,T]}$ is a martingale null at $t=0$ and orthogonal to the process $\widetilde P$ \footnote{The process $(M_t)_{t \in [0,T]}$
satisfies
\begin{align}
\ud M_t=& \widetilde P^{\CAT}_t \sqrt{1-\rho^2} \beta^{\CAT}_1(t, Z_t, L_t) \ud W^{1}_t \\
&+ \widetilde P^{\CAT}_{t^-}\int_{[0,+\infty)} \left(\alpha^{\CAT}(t, Z^1_t +\gamma^Z u,Z^2_t,L_{t^-}+u)-1\right) \left(m(\ud t, \ud z)-\nu (t,
Z^1_{t^-}, Z^2_{t},\ud u)\right).
\label{eq:O}
 \end{align}
}.%
Comparing equation \eqref{eq:GKW} and \eqref{eq:martingale_representation}   we obtain the optimal strategy
\begin{align}h^*_t=\theta(t, z_1, z_2, l)= \frac{p^{\CAT}_{z_2}(t,z_1,z_2,l)}{p_{z_2}(t,z_2)}+ \rho
\frac{\sigma_1}{\sigma_2}\frac{p^{\CAT}_{z_1}(t,z_1,z_2,l)}{p_{z_2}(t,z_2)},
\end{align}
and the hedging error $O_t=M_t$ for all $t\in [0,T]$.

Notice that $\theta$ is written in terms of the derivative of the functions $p$ and $p^{\CAT}$, which means that the computation of the hedging strategy requires that these functions are regular and hence a classical solution of the PDE \eqref{eq:bond-price} and the PIDE \eqref{eq:cat-pide}, respectively.

\section{Construction via change of measure}\label{sec:change}

We start from a probability space $(\Omega, \F, \widetilde{\P})$ with a filtration $\bF$, that supports a   $d$-dimensional-Brownian motion $W$ and
a Poisson random measure $ N(\ud t, \ud u)$ on $[0,T]\times E$ with $(\bF, \widetilde \P)$-compensator $\widetilde{\nu}(\ud u) \ud t$, where
$\widetilde{\nu}$ and $E$ are as in (A2) of Assumption \ref{ass:regularity} .  Let $X=(Z,L)$ be the unique strong solution to the following system
of SDEs
\begin{align}
&\ud Z_t= a(t,X_t) \ud t+ b(t,X_t) \ud W_t + \int_E \gamma^Z(t, X_{t^-}, u) N (\ud t, \ud u), \quad Z_0=z\in \R^d, \label{eq:Z}\\
&\ud L_t= \int_E \gamma^L(t, X_{t^-}, u)  N (\ud t, \ud u), \quad L_0=l\in \R, \label{eq:L}
\end{align}
where the functions $a$, $b$, $\gamma^Z$ and $\gamma^L$ satisfy (A0),  (A1) and (A3) in Assumption \ref{ass:regularity}.
The process $X$ is Markov under $\widetilde \P$ with  the generator
\begin{align}\label{eq:generator_Ptilde}
\widetilde{\mathcal{L}}_t \varphi(z,l)=&\sum_{i=1}^d a_i(t,z,l) \varphi_{z_i}(z,l)+\frac{1}{2}\sum_{i,j=1}^d \sigma_{i,j}(t,z,l) \varphi_{z_i,
z_j}(z,l) \\
&+  \int_E [\varphi(z+\gamma^Z(t,z,l,u), l+\gamma^L(t,z,l,u))-\varphi(z,l)]\widetilde \nu(\ud u),
\end{align}
for every $(z,l)\in  \R^d\times \R$ and every $t \in [0,T]$ and every function $(z,l) \to \varphi(z,l)$, $C^2$ in $z$ and continuous in $l$, bounded with bounded derivatives.

Using the Radon Nikodym density $\nu(t,x,u) = (\ud \nu(t,x)/\ud \widetilde{\nu})(u)$ introduced in (A2) of   Assumption \ref{ass:regularity},
we define the process  $\xi=(\xi)_{t \in [0,T]}$ as  the stochastic exponential
\begin{align}\label{eq:xi}
\xi_t=1+\int_0^t\xi_{s^-}\int_{E} \big ( \nu(s, X_{s^-},u) -1 \big ) (N(\ud s, \ud u)-\widetilde \nu(\ud u)\ud s), \quad t \in [0,T].
\end{align}
Applying the Dol\`{e}ans-Dade exponential formula we get that
\begin{align}\label{eq:Girsanov_density}
\xi_t=\prod_{T_n\leq t} \nu(T_n, X_{{T_n}-},U_n)  \exp\left(\int_0^t\int_{E}\left(1-\nu(s, X_{s^-},u)\right)\widetilde \nu(\ud u) \ud s\right),
\quad t \in [0,T],
\end{align}
where here $(T_n, U_n)_{n \geq 1}$ is the sequence of jump times and corresponding jump sizes of the measure $ N(\ud t, \ud u)$.

In the sequel we will need the following lemma.

\begin{lemma}\label{lemma:Girsanov}
The process $\xi=(\xi_t)_{t \in [0,T]}$ is bounded and satisfies $\xi_t\leq e^{\widetilde{\nu}(E) t}$ $\widetilde \P-a.s.$, for every $t \in [0,T]$. Let $\P$ be
the probability measure equivalent to $\widetilde{\P}$ defined by $\left.\frac{\ud \P}{\ud \widetilde{\P}}\right|_{\F_T}=\xi_T$. Then under $\P$, $
N(\ud t, \ud u)$ is a random measure with compensator  $\nu(t, X_{t^-}, \ud u)\ud t$ and $W$ is an $(\bF, \P)$-Brownian motion.
\end{lemma}

\begin{proof}
Since $\nu(t,x,u) \leq 1$ by (A2) in Assumption \ref{ass:regularity},  using the exponential form of $\xi$, we get that
\begin{align}
\xi_t\leq \exp\Big(\int_0^t \widetilde \nu(E) \ud s\Big) = e^{\widetilde \nu(E) t},\quad \P\mbox{-a.s.}, \ t \in [0,T].
\end{align}
Hence, the process $\xi$ is a true martingale as it is a bounded local martingale with $\widetilde{\mathbb{E}}\big [ \xi_T\big ] =1$, where
$\widetilde{\mathbb{E}}$ denotes the expectation under the probability measure $\widetilde \P$.
All the other claims follow directly from the Girsanov Theorem for marked point processes, see, e.g.  \citet[Theorem VIII.2]{bib:bremaud-81}.
\end{proof}

\begin{corollary}\label{cor-martingale-problem}
 Under Assumptions~\ref{ass:regularity} and \ref{ass:problem2}  there exists a unique solution of the martingale problem for the operator $\mathcal  L_t$ given in \eqref{eq:generator_P}.
\end{corollary}
\begin{proof}
{\em Existence.} By Lemma~\ref{lemma:Girsanov} we know that $W$ is an $(\bF, \P)$-Brownian motion and that the random measure $N(\ud u, \ud t)$ has the compensator $\nu(t, X_{t^-}, \ud u)\ud t$. Hence, for any function $\varphi:[0,T]\times \R^d\times \R\to \R$ which is $\mathcal C^1$ in $t$, $\mathcal C^2$ in $z$, continuous in $l$, bounded with bounded derivatives, by applying It\^o's lemma we get that
\[
\varphi(t,Z_t, L_t)=\varphi(0, Z_0, L_0)+\int_0^t\mathcal L_u\varphi(u,Z_u, L_u)+M_t
\]
where $M$ is the martingale given by
\begin{align*}
M_t=&\int_0^t b(s, X_s)\varphi_{z_i}(s, X_s) \ud W_s \\
&+ \int_0^t \int_E(\varphi(s,Z_{s^-}+\gamma^Z(s, X_{s^-},u), L_{s^-}+\gamma^L(s, X_{s^-},u))) (N(\ud u, \ud s)-\nu(s, X_{s^-}, \ud u)\ud s),
\end{align*}
so that under $\P$, the process $X = (Z,L)$ solves the martingale problem associated with $\mathcal{L}_t$.

{\em Uniqueness.} Here  we rely on  the well known result that the  martingale problem for $\mathcal{L}$ has a unique solution  if the marginal distributions of any solution process  $X$ are uniquely determined, see for instance \citet[Proposition 4.7 and Remark 4.8, Chapter 4]{bib:ethier-kurtz-86}. Now by
Theorem~\ref{thm:smooth_soln}, we get  existence of a smooth solution
to the backward PIDE with  $c(t,x)=0$ and $f(t,x)=0$ for all  bounded and Lipschitz continuous terminal conditions $g$, which immediately  implies uniqueness of the marginal distributions of $X$ and hence the result.
\end{proof}

\section{Proof of Theorem \ref{thm:smooth_soln}.}\label{sec:theorem}

This section is devoted to the proof of the main result of this paper. We denote $\widetilde \nu(E)=\widetilde \lambda$ and define the set
$\bar{D}\subseteq [0,T]\times \R^{d+1}\times \R^+$ as
\begin{align}
\bar{D}=\left\{(t,x,\xi)\in [0,T]\times \R^{d+1}\times \R^+: \xi \leq e^{\widetilde{\lambda}t}\right\}.
\end{align}
For $(t,x,\xi)\in \bar{D}$, let the function $\widetilde{v}$ be defined as
\begin{align}\label{eq:vtilde}
\widetilde{v}(t,x,\xi):=\widetilde{\mathbb{E}}\left[\xi_T\left(e^{-\int_t^Tc(s, X_s)\ud s}g(X_T) + \int_t^T e^{\int_t^s c(u, X_u)\ud u}f(s, X_s)\ud
s\right)|X_t=x, \xi_t=\xi\right],
\end{align}
where we recall that $\widetilde{\mathbb{E}}$ indicates the expectation under the probability measure $\widetilde{\P}$. By applying Bayes formula we
get that for every $(t,x,\xi)\in \bar{D}$,
\begin{align}
v(t,x)=\frac{\widetilde{v}(t,x,\xi)}{\xi}.
\end{align}
The reminder of the proof is organised as follows: in Step 1  we  consider the triple $\widetilde{X}=(Z,L,\xi)$ with the generator $\widetilde{\mathcal{L}}^{\widetilde{X}}$, and we use the results on viscosity solutions of \citet{bib:pham-98}
to show that $\widetilde{v}$ is a  Lipschitz continuous viscosity solution of  a backward equation involving the operator
$\widetilde{\mathcal{L}}^{\widetilde{X}}$. From this we can conclude in Step~2 that $v$ is
a  Lipschitz continuous viscosity solution of the {original} backward PIDE \eqref{eq:PDE_w_integral}. Finally, in Step~3  we  use a fixed
point argument to establish that $v$ is also a classical solution of that equation.

\subsection*{Step 1.}
First  we show that for $(t,x,\xi)\in \bar{D}$,
\begin{align} \label{eq:vtilde-2}
\widetilde{v}(t,x,\xi)=\widetilde{\mathbb{E}}\left[\xi_T e^{-\int_t^Tc(s, X_s)\ud s}g(X_T) +\int_t^T \xi_s e^{\int_t^s c(u, X_u)\ud u}f(s, X_s)\ud
s\mid X_t=x, \xi_t=\xi\right].
\end{align}
Indeed this follows from the sequence of  equalities
\begin{align}
&\widetilde{\mathbb{E}}\left[\xi_T \int_t^T e^{\int_t^s c(u, X_u)\ud u}f(s, X_s)\ud s|X_t=x, \xi_t=\xi\right] \\
&\qquad =\int_t^T \widetilde{\mathbb{E}}\left[\xi_Te^{\int_t^s c(u, X_u)\ud u}f(s, X_s)|X_t=x, \xi_t=\xi\right]\ud s\\
&\qquad =\int_t^T \widetilde{\mathbb{E}}\left[\xi_s e^{\int_t^s c(u, X_u)\ud u}f(s, X_s)|X_t=x, \xi_t=\xi\right]\ud s\\
&\qquad =\widetilde{\mathbb{E}}\left[\int_t^T \xi_s e^{\int_t^s c(u, X_u)\ud u}f(s, X_s)\ud s|X_t=x, \xi_t=\xi\right]
\end{align}
where we get the first and third equalities by applying the Fubini Theorem, since $\xi_T, c$ and $f$ are bounded, and the second equality follows from
the tower rule when conditioning on $\F_s$.

We now consider the triple $\widetilde{X}=(Z,L,\xi)$. Under Assumptions \ref{ass:regularity} the process $\widetilde X$ is a strong solution of the
system of SDEs \eqref{eq:Z}--\eqref{eq:L}--\eqref{eq:xi}, driven by an exogenous Poisson random measure. By Lemma \ref{lemma:Girsanov}, the process  $\xi$ is bounded and therefore we may consider the system on the state space $\bar{D}$.
Denote by $\widetilde{\mathcal{L}}^{\widetilde X}$ the $\widetilde{\P}$-Markov generator of the process $\widetilde{X}$. It holds that
\begin{align}
\widetilde{\mathcal{L}}^{\widetilde X}_t \varphi(z,l,\xi)&=\sum_{i=1}^d a_i(t,z,l) \varphi_{z_i}(z,l,\xi)+\frac{1}{2}\sum_{i,j=1}^d
\sigma_{i,j}(t,z,l) \varphi_{z_i, z_j}(z,l,\xi) \\
& - \Big ( \xi \int_E  \left(\nu(t, z,l,u) - 1\right) \widetilde \nu(\ud u) \Big ) \,  \varphi_{\xi}(z,l,\xi)\\
&+    \int_{E} \left[\varphi\left(z+\gamma^Z(t,z,l,u), l+\gamma^L(t,z,l,u), \xi \nu(t, z,l ,u) \right)-\varphi(z,l, \xi)\right]   \widetilde \nu(\ud
u)
\end{align}
for every $(t,z,l,\xi)\in \bar{D}$ and  for every function $(z,l,\xi)\to \varphi(z,l,\xi)$ which is bounded, $\mathcal{C}^2$ in $z$, continuous in
$l$ and $\mathcal{C}^1$ in $\xi$. Here $\varphi_{\xi}$ indicates the first derivative of $\varphi$ with respect to $\xi$.

Note that the system  \eqref{eq:Z}--\eqref{eq:L}--\eqref{eq:xi} satisfies Conditions (2.1)--(2.6) in  \citet{bib:pham-98}. Indeed, Conditions (2.1),
(2.2), (2.5) and (2.6) follow directly from  our assumptions (A.1), (A.2), (A.4) and (A.5);  Conditions (2.3) and (2.4) follow from (A.2), (A.3) and
from the fact that on $\bar{D}$ the mapping
\begin{align}
(t,x,\xi)\to\xi\left(\nu(t,x,u) - 1\right)
\end{align}
is bounded and Lipschitz, as it is a product of two bounded Lipschitz functions.
Therefore we can now apply  \citet[Theorem 3.1 and Proposition 3.3]{bib:pham-98} and get that the function $\widetilde v$ in equation
\eqref{eq:vtilde} is continuous in $\bar{D}$ and  Lipschitz in $(x,\xi)$, uniformly in $t$ (i.e. $\widetilde{v} \in W^1(\bar D)$). Moreover
$\widetilde v$ is a viscosity solution of the backward equation
\begin{align}
&\widetilde v_t(t,z,l,\xi)+\widetilde{\mathcal{L}}^{\widetilde{X}}_t \widetilde{v} (t,z,l,\xi)=c(t,x) \widetilde{v}(t,z,l,\xi)+\xi f(t,z,l), \qquad (t,z,l,\xi)
\in \bar D,\\
&\widetilde v(T,z,l,\xi)=g(l,z)\xi, \qquad(z,l,\xi)\in \R^{d}\times \R \times [0, e^{\widetilde \lambda T}].
\end{align}
Note that here we need  the alternative representation  \eqref{eq:vtilde-2} for $\widetilde v$ to get the dividend term $\xi f(t,z,l)$.

\subsection*{Step~2.} Let $\phi:[0,T]\times \R^d \times \R \to \R$ be a smooth function and define the function $\widetilde \phi: \bar D \to R$ by
$\widetilde \phi(t,z,l,\xi)=\phi (t,z,l)\xi$. Then for every $(t,z,l,\xi)\in \bar{D}$ we have that
\begin{align}
\widetilde{\mathcal{L}}^{\widetilde X}_t \ \widetilde \phi (t,z,l,\xi) &=
\xi \bigg\{\sum_{i=1}^d a_i(t,z,l) \phi_{z_i}(t,z,l)+\frac{1}{2}\sum_{i,j=1}^d \sigma_{i,j}(t,z,l) \phi_{z_i, z_j}(t,z,l)  \\&+  \int_E [\phi(t,
z+\gamma^Z(t,z,l,u), l+\gamma^L(t,z,l,u))-\phi(t,z,l)]\nu(t, z,l;\ud u)\bigg\},
\end{align}
and this is of course equal to $\xi \mathcal{L}_t\phi (t,z,l).$
Consequently we see that $v$ is a viscosity solution of the {\em original} backward PIDE \eqref{eq:PDE_w_integral}.

\subsection*{Step~3.} We finally want to show that function $v$ is a classical solution of the backward PIDE~\eqref{eq:PDE_w_integral} and hence, in
particular,  that $v$ is $\mathcal C^1$ in $t$ and $\mathcal C^2$ in $z$. For this we modify the  fixed point argument used in the proof of
\citet[Proposition~5.3]{bib:pham-98}\footnote{Notice that \citet[Proposition~5.3]{bib:pham-98} cannot be applied directly  to the function $\widetilde v$, since that result requires uniform ellipticality of the diffusion coefficient of $X$ (which is not satisfied in our setup due to the presence of the pure jump component $L$).}.
 Define for a bounded  function $\varphi:[0,T]\times \R^d\times \R \to \R$ that is Lipschitz in $(z,l)$,
uniformly in $t$,  the function $F_{[\varphi]}$ as
\begin{align}
F_{[\varphi]}(t,z,l):= \int_E [\varphi(t, z+\gamma^Z(t,z,l,u), l+\gamma^L(t,z,l,u))-\varphi(t,z,l)]\nu(t, z,l;\ud u).
\end{align}
Using (A2) and (A3) in Assumption \ref{ass:regularity} it is easily seen that $F_{[\varphi]}$ is Lipschitz in $(z,l)$, uniformly in $t$.

Recall the definition of the differential operator $\mathcal{L}^*_{(t,l)}$ from equation \eqref{eq:generator*}.
It follows from \citet[Lemma~2.1]{bib:pham-98}  that  $v$ is also a  viscosity solution of the backward equation
\begin{align}\label{eq:PDE_wo_integral}
&v_t(t,z, l)+\mathcal{L}^*_{t, l} v(t,z, l)+F_{[v]}(t,z, l)=c(t,z,l)v(t,z, l)+f(t,z, l), \quad (t,z)\in [0,T]\times \R^d\\
&v(T,z, l)=g(z,  l), \qquad z\in \R^d,
\end{align}
(see \citet[page 22]{bib:pham-98} for details). Note that  equation \eqref{eq:PDE_wo_integral} is a linear parabolic  partial differential equation
(and not a PIDE), and   a bounded  classical solution  $u(t,z,l)$ to~\eqref{eq:PDE_wo_integral} exists by Assumption~\ref{ass:problem2}. Now  $u$ is
clearly also a viscosity solution of  \eqref{eq:PDE_wo_integral}. Uniqueness results for  viscosity solutions  of linear parabolic PDEs imply
that $u=v$ and the regularity of $v$  follows.

Finally, uniqueness of classical solutions of \eqref{eq:PDE_w_integral} follows, for instance, from  \citet[Proposition 5.2]{bib:pham-98}.


\begin{thebibliography}{12}
\providecommand{\natexlab}[1]{#1}
\providecommand{\url}[1]{\texttt{#1}}
\expandafter\ifx\csname urlstyle\endcsname\relax
  \providecommand{\doi}[1]{doi: #1}\else
  \providecommand{\doi}{doi: \begingroup \urlstyle{rm}\Url}\fi

\bibitem[Barles and Souganidis(1991)]{barles1991}
G.~Barles and P.~E. Souganidis.
\newblock {Convergence of approximation schemes for fully nonlinear second
  order equations}.
\newblock \emph{Asymptotic Analysis}, 4:\penalty0 271--283, 1991.


\bibitem[Bensoussan and Lions(1982)]{bib:bensoussan-lions-82}
A.~Bensoussan and J.L. Lions.
\newblock \emph{Impulse Control and Quasi Variational Inequalities}.
\newblock Dunod, Paris, 1982.

\bibitem[Bielecki and Rutkowski(2002)]{bib:bielecki-rutkowski-02a}
T.~R.~Bielecki and M.~Rutkowski.
\newblock \emph{Credit Risk: Modeling, Valuation, and Hedging}.
\newblock Springer, Berlin, 2002.

\bibitem[Br{\'e}maud(1981)]{bib:bremaud-81}
P.~Br{\'e}maud.
\newblock \emph{Point Processes and Queues: Martingale Dynamics}.
\newblock Springer, New York, 1981.

\bibitem[Cartea et~al.(2015)Cartea, Jaimungal, and
  Penalva]{bib:cartea-jaimungal-penalva-15}
{\'A}.~Cartea, S.~Jaimungal, and J.~Penalva.
\newblock \emph{Algorithmic and High-Frequency Trading}.
\newblock Cambridge University Press, 2015.


\bibitem[Ceci et~al.(2015)Ceci, Colaneri, Cretarola]{bib:ceci-hedging2015} {C}. Ceci, {K}. Colaneri, and {A.} Cretarola.
\newblock Hedging of unit-linked life insurance contracts with unobservable mortality hazard rate via local risk-minimization. \newblock
\emph{Insurance: Mathematics and Economics}, 60:\penalty0, 47--60, 2015.




\bibitem[Ceci et~al.(2019)Ceci, Colaneri, Frey, K\"ock]{bib:ceci-colaneri-frey-koeck-20} {C}.~Ceci, {K}.~Colaneri, R.~Frey and V.~K\"ock.
\newblock Value adjustments and dynamic hedging of reinsurance counterparty credit risk.
\newblock \emph{ SIAM Journal on Financial Mathematics}, 11\penalty0 (3):\penalty0 788--814,  2020.


\bibitem[Colaneri et~al.(2019)Colaneri, Eksi, Frey, Sz\"olgyenyi]{bib:colaneri-eksi-frey-szolgyenyi-19}{K}.~Colaneri, {Z}.~Eksi, R.~Frey and {M}.
    Sz\"olgyenyi.
\newblock Optimal liquidation under partial information with price impact.
\newblock \emph{Stochastic processes and their applications}, 130\penalty0 (4):\penalty0, 1913--1946, 2020.

\bibitem[Cordoni and Di~Persio(2013)]{cordoni2013transition}
F.~Cordoni and L.~Di~Persio.
\newblock Transition density for CIR process by Lie symmetries and application
  to ZCB pricing.
\newblock \emph{International Journal of Pure and Applied Mathematics},
  88\penalty0 (2):\penalty0 239--246, 2013.


\bibitem[Cox and Pedersen(2000)]{bib:cox-2000}S.H. Cox, and H.W. Pedersen
\newblock Catastrophe risk bonds.
\newblock \emph{North American Actuarial Journal}, 4\penalty0 (4):\penalty0, 56--82, 2000.

\bibitem[Davis and Lleo(2013)]{bib:davis-lleo-13}
M.~Davis and S.~Lleo.
\newblock Jump-diffusion risk-sensitive asset management II: jump-diffusion
  factor model.
\newblock \emph{SIAM Journal on Control and Optimization}, 51\penalty0
  (2):\penalty0 1441--1480, 2013.

\bibitem[Eymen et ~al.(2010)]{bib:giesecke} E.~Eymen, K.~Giesecke and L.~R.~Goldberg. \newblock Affine point processes and portfolio credit risk.
    \newblock \emph{SIAM Journal on Financial Mathematics} 1.1:\penalty0 642--665, 2010.

\bibitem[Filipovic (2009)]{bib:filipovic-2009} D. Filipovic.
\newblock \emph{Term-Structure Models. A Graduate Course}.
\newblock Springer, Berlin-Heidelberg, 2009.


\bibitem[Ethier and Kurtz(1986)]{bib:ethier-kurtz-86}
S.~Ethier and T.~G. Kurtz.
\newblock \emph{Markov Processes: Characterization and Convergence}.
\newblock Wiley, New York, 1986.

\bibitem[F\"{o}llmer and Sondermann(1986)]{bib:follmer-86}
H. F\"{o}llmer, and D. Sondermann.
\newblock {Hedging of non-redundant contingent claims.}
\newblock In W. Hildenbrand, and A. Mas-Colell, editors, \emph{ Contributions to mathematical economics}, pages 205--223. Amsterdam, North-Holland,
1986.


\bibitem[Frey(2000)]{bib:frey-00b}
R.~Frey.
\newblock Risk-minimization with incomplete information in a model for high
  frequency data.
\newblock \emph{Mathematical Finance}, 10\penalty0 (2):\penalty0 215--225,
  2000.

\bibitem[Frey and Runggaldier(2001)]{bib:frey-runggaldier-01}
R.~Frey and W.~Runggaldier.
\newblock A~nonlinear filtering approach to volatility estimation with a~view
  towards high frequency data.
\newblock \emph{International Journal of Theoretical and Applied Finance},
  4:\penalty0 199--210, 2001.

\bibitem[Friedman (2008)]{bib:friedman-64} A.~Friedman, \newblock \emph{Partial Differential Equations of Parabolic Type}. Courier Dover
    Publications, 2008.

\bibitem[Gihman and Skohorod(1980)]{bib:gihman-skohorod-80}
I.~Gihman and A.~Skohorod.
\newblock \emph{The {T}heory of {S}tochastic {P}rocesses}, volume {I}{I}{I}.
\newblock Springer, New York, 1980.

\bibitem[Grandell(1991)]{bib:grandell-91}
J.~Grandell.
\newblock \emph{Aspects of {R}isk {T}heory}.
\newblock Springer, Berlin, 1991.

\bibitem[Jaimungal and Chong(2014)]{bib:jaimungal-14}S. Jaimungal and Y. Chong. \newblock \emph{Valuing clustering in catastrophe derivatives}.
    \newblock Quantitative Finance, 14\penalty0 (2):\penalty0 259--270, 2014.

\bibitem[Jarrow(2010)]{bib:jarrow2010}
R.A.~Jarrow.
 A simple robust model for CAT bond valuation. \emph{Finance {R}esearch {L}etters},
 7\penalty0 (2):\penalty0 72--79, 2010.

\bibitem[Lee and Yu(2002)]{bib:lee-2004} J.P. Lee and M.T. Yu.
\newblock Pricing default-risky CAT bonds with moral hazard and basis risk.
\newblock \emph{ Journal of Risk and Insurance}, 25--44, 2002.

\bibitem[McNeil(2015)]{bib:McNeil-2015}
A.J.~McNeil, R.~Frey and P.~Embrechts
\newblock \emph{Quantitative risk management: concepts, techniques and tools.} Revised edition.
\newblock Princeton university press, 2015.



\bibitem[Pham(1998)]{bib:pham-98}
H.~Pham.
\newblock {O}ptimal stopping of controlled jump diffusion processes: a viscosity
  solution approach.
\newblock \emph{Journal of {M}athematical {S}ystems, {E}stimation and {C}ontrol},
  8:\penalty0 1--27, 1998.


\bibitem[Schweizer(2001)]{schweizer2001guided}
M.~Schweizer.
\newblock A guided tour through quadratic hedging approaches.
\newblock In E.~Jouini, J.~Cvitanic, and M.~Musiela, editors, \emph{Option
  Pricing, Interest Rates and Risk Management}, pages 538--574. Cambridge
  University Press, 2001.


\end{thebibliography}
\end{document}